\documentclass[11pt, reqno]{amsart}

\usepackage{amsmath, amsthm, amssymb}
\usepackage{graphicx}
\usepackage{hyperref}
\usepackage{tikz-cd}
\usepackage{color}
\usepackage{enumitem}
\usepackage[all,cmtip]{xy}
\usepackage[all]{xy}
\newtheorem{theorem}{Theorem}[section]
\newtheorem{lemma}[theorem]{Lemma}

\theoremstyle{definition}
\newtheorem{definition}[theorem]{Definition}
\newtheorem{example}[theorem]{Example}
\newtheorem{remark}[theorem]{Remark}

\newcommand{\cl}{\mathrm{Cl}}
\newcommand{\R}{\mathbb{R}}
\newcommand{\Z}{\mathbb{Z}}
\newcommand{\C}{\mathcal{C}}

\newcommand{\V}{\mathcal{V}}
\newcommand{\con}[2]{\mathrm{con}_{#1}(#2)}
\newcommand{\ex}{\mathrm{Ex}}

\title[Component-Perfect MDM Functions ]{Component-Perfect Multidimensional Discrete Morse Functions}

\author{Mehmetcik Pamuk}
\address{Department of Mathematics, Middle East Technical University, Ankara 06531, Turkey}
\email{mpamuk@metu.edu.tr}

\author{Hanife Varl\i}
\address{Department of Mathematics, Çankırı Karatekin University, Çankırı 18100, Turkey}
\email{hanifevarli@karatekin.edu.tr}

\date{\today}

\begin{document}

\begin{abstract}
We introduce the concept of component-perfect multidimensional discrete Morse (MDM) functions on finite regular CW complexes. While perfect discrete Morse functions in the classical sense require the number of critical cells of index $p$ to exactly equal the $p$-th Betti number, the multiparameter framework necessitates a topological dynamics perspective, where critical cells naturally cluster into connected components. Building upon the existing notion of strictly perfect MDM functions, we propose the more flexible definition of component-perfection.

We also show how to restrict and extend MDM functions on finite regular CW complexes, laying the groundwork for more complex topological operations.
\end{abstract}

\maketitle

\section{Introduction}
\label{sec:intro}

Since its introduction by Marston Morse in the 1920s, Morse theory has been a 
powerful tool in the study of smooth manifolds, describing their topology through 
the critical points of smooth functions defined on them. In the 1990s, Robin Forman 
developed a discrete analogue of this theory \cite{Forman1998, Forman2002} that has 
become a foundational tool in combinatorial topology and topological data analysis 
(TDA). A discrete Morse function on a finite cell complex $K$ assigns a real number 
to each cell in such a way that the natural order given by cell dimension is 
respected, except at most in one face or coface for each cell. The cells that 
violate no such exception are the \emph{critical cells}, and they control the 
topology of $K$ via the discrete Morse inequalities: the number of critical 
$p$-cells is at least the $p$-th Betti number $\beta_p(K)$.

A discrete Morse function achieving this lower bound, having exactly $\beta_p(K)$ 
critical $p$-cells for each $p$, is called \emph{perfect}. Perfect discrete Morse 
functions are the most computationally desirable, providing minimal cell 
decompositions of the underlying space and enabling efficient homology computation 
\cite{MischaikowNanda2013}. However, they do not always exist: torsion in homology 
is an obstruction to $\Z$-perfection, and certain acyclic but non-collapsible 
complexes such as the dunce hat and Bing's house admit no perfect discrete Morse 
function at all. Even on triangulated manifolds, finding a perfect discrete Morse 
function is NP-hard in general \cite{JoswigPfetsch2006}.

Despite these obstructions, large and natural classes of manifolds do admit perfect 
discrete Morse functions. Every closed oriented surface admits a $\Z$-perfect 
discrete Morse function \cite{LewinerLopes2003}, and every closed surface admits a 
$\Z_2$-perfect one \cite{AyalaVilches2012}. In dimension 3, any manifold of the 
form $F \times S^1$, where $F$ is a closed oriented surface, admits a $\Z$-perfect 
discrete Morse function, as do their connected sums.

The behavior of perfect discrete Morse functions under topological surgeries, such as the connected sum operation on closed oriented manifolds, has been a fruitful area of study in the classical setting \cite{Varl2018,MramorKosta2018}. Extending such geometric operations to the multiparameter setting requires a highly flexible notion of perfection, which motivates the theoretical framework developed in this paper.

\subsection*{Multiparameter Persistence and Multidimensional Discrete Morse Theory}

In many modern applications of TDA, data is most naturally parametrized by 
multiple parameters simultaneously, for instance, scale and density in point cloud 
analysis, or the presence of multiple noise sources in biological imaging 
\cite{VipEtAl2021}. Single-parameter persistent homology, while powerful, can miss 
topological structure that only appears when several parameters vary together. This 
has driven the development of multiparameter persistent homology (MPH), 
defined on filtrations indexed by $\R^k$ for $k \geq 2$ 
\cite{CarlssonZomorodian2009}. One of the main computational challenges in MPH is 
the size of the complexes involved; reducing them while preserving homotopy type is 
therefore a central goal.

Discrete Morse theory provides a natural reduction strategy: an acyclic partial 
matching on a simplicial complex $K$ collapses non-critical pairs and leaves a 
smaller complex homotopy equivalent to $K$. Extending this strategy to the 
multiparameter setting requires a corresponding extension of discrete Morse theory 
itself to vector-valued functions $f: K \to \R^k$. A first step was taken by Allili, 
Kaczynski, Landi, and Masoni \cite{Allili2019}, who defined multidimensional discrete 
Morse (MDM) functions, established the existence of a gradient vector field, and 
developed algorithms for acyclic partial matchings in the multiparameter setting.     
A systematic and complete theoretical development of MDM theory was then carried out 
by Brouillette, Allili, and Kaczynski \cite{Brouillette2024}.  They established, among 
other results, the handle decomposition and collapsing theorems, and derived Morse 
inequalities and Morse decompositions in the multiparameter setting.

A key feature distinguishing MDM theory from the classical ($k=1$) case is the 
nature of critical points. While a real-valued discrete Morse function has isolated critical cells, an MDM function generically produces critical cells that cluster into connected critical components, in close analogy with the Pareto sets that appear in smooth singularity theory for vector-valued functions \cite{BudneyKaczynski2023, Smale1975, Wan1975}.
This clustering is not a pathology 
but a geometric feature: it reflects the multi-dimensional structure of the 
codomain and is essential for capturing the richer topological information that 
multiparameter persistence encodes.

To formally handle these clusters, Brouillette et al.\ \cite{Brouillette2024} 
partition critical cells into equivalence classes under 
dynamical connectivity. Because the multidimensional discrete gradient flow $\Pi_f$ 
is inherently multivalued, these critical components cannot be analyzed via single 
isolated cells. Instead, the theory treats each entire critical component as an 
isolated invariant set with respect to the multiparameter flow. In the 
framework of combinatorial topological dynamics, a set of cells $S$ is an 
isolated invariant set if it admits an isolating neighborhood, a closed subcomplex 
$N$ containing $S$ such that any gradient flow path lying entirely within $N$ is 
necessarily contained entirely within $S$ \cite{Batko2020}. This topological 
isolation ensures that the component acts as a single, indivisible dynamical entity 
from the perspective of the flow. Brouillette et al.\ prove that, under an acyclicity 
condition on the MDM function, these isolated invariant sets induce a well-defined 
Morse decomposition of the complex, yielding generalized Morse inequalities expressed 
in terms of their discrete Conley indices \cite{Brouillette2024}.

 In classical discrete Morse theory, perfect functions are highly sought after because they yield minimal cell complexes, optimizing homology computations \cite{MischaikowNanda2013}.  Computing multiparameter persistent homology presents significant computational challenges, primarily due to the combinatorial explosion in the size of the associated algebraic complexes \cite{CarlssonZomorodian2009}. A perfect MDM function provides the  theoretical bound for topological data reduction: it compresses the multiparameter filtration down to the absolute minimal number of homological generators required by the manifold, stripping away all trivial artifacts without altering the underlying homotopy type.

Our paper bridges the two lines of research described above. We  
study the notion of a perfect MDM function, extending the classical concept 
to the multiparameter setting. The multi-valued nature of the critical point theory 
necessitates two distinct definitions of perfection:

\begin{itemize}
    \item A \emph{strictly perfect} MDM function has exactly $\beta_p(K)$ critical cells of index $p$ for each $p$. This is the direct analogue of classical perfection, but we show it is a highly rigid and non-generic condition.
    \item A \emph{component-perfect} MDM function satisfies $m_p = \beta_p(K)$ for 
    all $p$, where $m_p$ is the sum of the $p$-th Conley coefficients of the 
    critical components. This is the natural and flexible notion of perfection in the 
    multiparameter setting, allowing critical components to consist of multiple 
    cells provided their collective Conley index matches the topological 
    requirement.
\end{itemize}

We prove that strictly perfect implies component-perfect but not conversely, and 
discuss when each class is non-empty. 

The paper is organized as follows. In Section~\ref{sec:prelims}, we recall the 
necessary background on MDM functions, gradient fields, Morse decompositions, and 
Conley indices, following \cite{Allili2019} and \cite{ Brouillette2024}. In addition, we establish the 
foundational restriction and extension properties for an MDM function. In 
Section~\ref{sec:perfect_def}, we introduce both notions of perfection and discuss 
their relationship and existence. In a subsequent paper, we will build upon this theoretical framework to explicitly construct and decompose component-perfect MDM functions on connected sums.

\section{Preliminaries}
\label{sec:prelims}
In this section, we recall necessary basic notions of discrete and multidimensional discrete Morse theory.
\subsection{Regular cell complexes}
 Let $K$ be a finite regular CW complex and $K_p$ be the set of open cells of dimension $p$, briefly called $p$-cells, for any $p\geq 0$. For any cell $\sigma\in K_p$ and $\tau\in K_{q}$ for $p\leq q$, we denote $\sigma\subset \tau$ if the closure of $\tau$ contains $\sigma$, and we say $\sigma$ is a face of $\tau$ or $\tau$ is a coface of $\sigma$. 

Let $S$ be a subset of cells in $K$.
The closure of $S$, denoted $\cl(S)$, is the smallest 
subcomplex of $K$ containing $S$; equivalently,
$\cl(S) = \bigl\{\sigma \in K \;\bigm|\;
\sigma \subseteq \tau
\text{ for some } \tau \in S \bigr\}$.  
The exit set of $S$ is $\mathrm{Ex}(S) = \cl(S) \setminus S$.

Let $\sigma^{(p)}, \tau^{(p+1)}$ be $p$- and $(p+1)$-cells in $K$ such that $\sigma \subset \tau$. If there is no other $(p+1)$-cell $\tau'\in K$ such that $\sigma \subset \tau'$, then we say that $\sigma$ is a free face of $\tau$. In this case, we call the pair $(\sigma, \tau)$ a collapse pair in $K$. Let $L$ be a subcomplex of $K$. If $L$ is obtained by removing a finite sequence of collapse pairs from $K$, then we say that $K$ collapses to $L$.

\subsection{Discrete Morse Functions}\label{DMF}
Let $K$ be a finite regular CW complex of dimension $d$ and  $K_p$ be the set of open $p$-cells of $K$. Let $f\colon K\rightarrow \mathbb{R}$ be a real valued function on $K$. For any $p$-cell $\sigma \in K_p$, consider the sets $$H_f(\sigma)=\{ \tau \in K_{p+1} \  | \ \tau \supset \sigma  \text{ and } f(\tau)\leq f(\sigma) \}$$ and $$T_f(\sigma)=\{ \nu \in K_{p-1} \ | \ \nu \subset \sigma \text{ and } f(\nu)\geq f(\sigma) \}.$$ 

We say $f$ is a discrete Morse function if for any $p$-cell 
$\sigma \in K$, $|H_f(\sigma)|\leq 1$ and $|T_f(\sigma)|\leq 1$, where $|\cdot|$ denotes the cardinality of the related set. 

Observe that the numbers $|H_f(\sigma)|$ and $|T_f(\sigma)|$ cannot be both one. If  $|H_f(\sigma)|$ and $|T_f(\sigma)|$ are both zero, $\sigma$ is called a critical $p$-cell. Otherwise, it is called a regular cell. 

In other words, a discrete Morse function on $K$ is a function whose values
increase with the dimension of cells, with at most one exception per cell.

Let us consider the collection of pairs of cell given in the set $$
V = \{(\sigma, \tau) \mid \dim \sigma=\dim \tau-1, \sigma \subset \tau\}.
$$ If each cell in $K$ is at most one pair in $V$, then $V$ is called a discrete vector field on $K$. Moreover, if $f(\sigma)\geq f(\tau)$ for every pair $(\sigma, \tau)$ in the discrete vector field $V$, then $V$ is called the discrete gradient vector field of $f$. 

Recall that the number of critical $p$-cells of a discrete Morse function $f$ is always greater than or equal to the $p$-th Betti number, $\beta_p(K)$, of $K$ (with respect to the given coefficient ring) \cite{Forman1998}. A classical discrete Morse function is called perfect (with respect to the given coefficient ring) if the number of critical $p$-cells exactly equals $\beta_p(K)$ for all $p$.

\subsection{Multidimensional Discrete Morse Functions}
We denote by $\preceq$  the partial order on
$\mathbb{R}^k$ defined component-wise: $a \preceq b$ if and only if $a_i \leq b_i$ for
all $i = 1, \dots, k$. We write $a \prec b$ when $a \preceq b$ and $a \neq b$. 

Allili et al.\ \cite{Allili2019} defines the multidimensional discrete Morse function for finite simplicial complexes in \cite[Definition~11]{Allili2019}.  However,  we observe that this definition of an MDM function extends to a finite regular CW complex $K$ without further modifications. 
But we first state a useful property of regular CW complexes given in \cite[Theorem~1.3]{Forman1998}. 

 \begin{theorem}\label{thm:property_regularCW}
 Let $K$ be a regular CW complex and $\nu^{(p-1)}, \sigma^{(p)}, \tau^{(p+1)}\in K$. If $\nu\subset \sigma\subset \tau$, then there exists a $p$-cell $\tilde{\sigma}$ such that $\tilde{\sigma}\neq \sigma$ and $\nu\subset \tilde{\sigma} \subset \tau$. 
\end{theorem}

Let $K$ be a finite regular CW complex. For $f = (f_1, \dots, f_k): K \to \mathbb{R}^k$ and
$\sigma \in K_p$, define
\begin{align*}
H_f(\sigma) &= \{ \beta \in K_{p+1} \mid \beta \supset \sigma,\ f(\beta) \preceq f(\sigma) \}, \\
T_f(\sigma) &= \{ \alpha \in K_{p-1} \mid \alpha \subset \sigma,\ f(\alpha) \succeq f(\sigma) \}.
\end{align*}

\begin{definition} \label{def:mdm_function_on_regularCW}
A function $f = (f_1, \dots, f_k): K \to \mathbb{R}^k$ is a
multidimensional discrete Morse (MDM) function if for every $p$-cell in $K_p$:
\begin{enumerate}
    \item $|H_f(\sigma)| \leq 1$,
    \item $|T_f(\sigma)| \leq 1$,
    \item if $\beta^{(p+1)} \supset \sigma$ and $\beta \notin H_f(\sigma)$,
          then $f(\beta) \succ f(\sigma)$,
    \item if $\alpha^{(p-1)} \subset \sigma$ and $\alpha \notin T_f(\sigma)$,
          then $f(\alpha) \prec f(\sigma)$.
\end{enumerate}
A cell $\sigma$ is critical if $|H_f(\sigma)| = |T_f(\sigma)| = 0$,
and regular otherwise.
\end{definition}

The following Lemma is a straightforward extension of \cite[Proposition~12]{Allili2019} to finite regular CW complexes
because of the property given in Theorem \ref{thm:property_regularCW}. 
\begin{lemma}\label{lem:product}
Let $f: K \to \mathbb{R}^k$ be an MDM function on a finite regular CW complex $K$. Then $|H_f(\sigma)||T_f(\sigma)|=0$ for each $p$-cell $\sigma\in K$.
\end{lemma}
\begin{proof}
Suppose, for contradiction, that $|H_f(\sigma)| = |T_f(\sigma)| = 1$
for some $p$-cell $\sigma$.  Let $\beta \in H_f(\sigma)$ (a
$(p+1)$-cell with $\beta \supset \sigma$ and $f(\beta) \preceq
f(\sigma)$) and $\alpha \in T_f(\sigma)$ (a $(p-1)$-cell with
$\alpha \subset \sigma$ and $f(\alpha) \succeq f(\sigma)$).  Then
$\alpha \subset \sigma \subset \beta$. So, by
Theorem~\ref{thm:property_regularCW}, there exists another $p$-cell
$\tilde{\sigma} \neq \sigma$ such that $\alpha \subset \tilde{\sigma}
\subset \beta$. Then, by Definition \ref{def:mdm_function_on_regularCW}, we obtain $f(\alpha) \prec f(\tilde{\sigma}) \prec
f(\beta) \preceq f(\sigma) \preceq f(\alpha),$ which is a contradiction.
\end{proof}
  
The following definitions are direct extensions of the corresponding definitions given in \cite{Allili2019} to the finite regular CW complexes. 

\begin{definition}\label{def:gradient_mdm}
Let $K$ be a finite regular CW complex and $f : K \to \mathbb{R}^k$ be an MDM function. The gradient vector
field of $f$ is the discrete vector field
\[
  \mathcal{V} = \bigl\{(\sigma,\, \beta) \;\bigm|\;
  \dim\beta = \dim\sigma + 1,\; \sigma \subset \beta,\;
  f(\beta) \preceq f(\sigma)\bigr\}.
\]
The fixed points of $\mathcal{V}$ (cells not appearing in any pair) are
precisely the critical cells of $f$.
\end{definition}

\begin{remark}
Conditions (1) and (2) are direct analogues of the classical Forman conditions. Conditions (3) and (4) are specific to the multiparameter setting: they ensure that the values of $f$ at facets and cofacets outside the gradient pairing are strictly comparable, which is needed to guarantee acyclicity of the gradient vector field. Note that when $k=1$, conditions (3) and (4) are automatically satisfied, and the definition reduces to Forman's.
\end{remark}

\begin{definition}\label{def:V-path}
Let $\mathcal{V}$ be a gradient vector field of an MDM function $f : K \to \mathbb{R}^k$ on a finite regular CW complex $K$. The sequence 
$$
\sigma_0\subset\beta_0\supset\sigma_1\subset\beta_1\supset\ldots\supset\sigma_n\subset\beta_{n}\supset\sigma_{n+1}
$$ is called a gradient path of dimension $p$ (shortly, a $p$-path) of $\mathcal{V}$ for $p\ge 1$ if $\dim \sigma_i=p-1$, $\dim \beta_i=p$, $(\sigma,\, \beta)\in \mathcal{V} $ and $\sigma_i\neq \sigma_{i+1}$ for $0\leq i \leq n$.     
\end{definition}

The following is an example of an MDM funtion on a $2$-dimensional simplicial complex $K$ and its gradient vector field. 
\begin{example}
\label{ex:MDMfunction} Consider the two parameter function 
 $f:K\to \mathbb{R}^2$ given in Figure \ref{fig:MDMfunction}. Clearly, $f$ is a MDM function whose critical cells are the ones in red.  To explicitly verify that these cells are critical, let us apply Definition~\ref{def:mdm_function_on_regularCW} to two examples from this set:
\begin{itemize}[leftmargin=*]
    \item For the vertex $f^{-1}((0,1))$: Checking its cofacets, we see the horizontal edge has value $(2,1)$ and the vertical edge has value $(3,2)$. Since neither $(2,1) \preceq (0,1)$ nor $(3,2) \preceq  (0,1)$ is true, the sublevel set of cofacets is empty ($|H_f(f^{-1}((0,1)))| = 0$). Because vertices have no facets, $|T_f(f^{-1}((0,1)))| = 0$. Thus, $v_1$ is critical.
    \item For the edge $f^{-1}((5,4))$: Checking its facets, they have values $(5,3)$ and $(2,2)$. Since neither $(5,3) \succeq (5,4)$ nor $(2,2) \succeq (5,4)$ is true, we have $|T_f(f^{-1}((5,4)))| = 0$. Its only cofacet in the complex has a value of $(5,7)$, and since $(5,7) \not\preceq (5,4)$, we have $|H_f(f^{-1}((5,4)))| = 0$. Thus, $f^{-1}((5,4))$ is also critical.
\end{itemize} 
\begin{figure}[htb]
\includegraphics[width=0.4\textwidth]{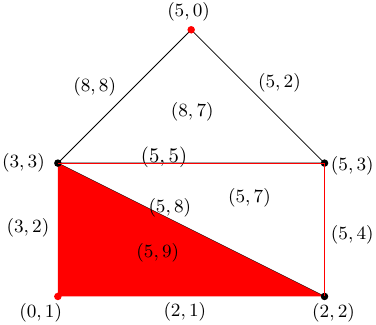}
\hspace{10mm}
\includegraphics[width=0.4\textwidth]{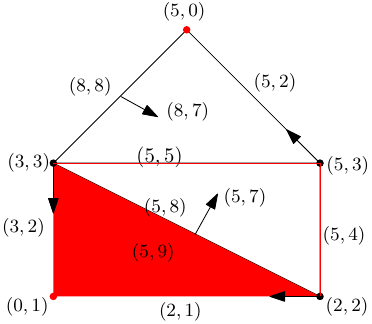}
\caption{A MDM function and its gradient vector field.}
\label{fig:MDMfunction}
\end{figure}
\end{example}

Now we give a useful property of a gradient vector field of an MDM function. 

\begin{lemma}\label{lem:1merge}
Let $K$ be a finite regular CW complex and $\mathcal{V}$ be a gradient vector field of an MDM function $f : K \to \mathbb{R}^k$. Then each $1$-path of $\mathcal{V}$ ends at a critical $0$-cell. Moreover, $1$-paths of $\mathcal{V}$ can merge but cannot split. 
\end{lemma}

\begin{proof}
Let $\sigma_0 \subset \beta_0 \supset \sigma_1 \subset \dots \supset \sigma_n$ be a $1$-path in $\mathcal{V}$, where the $\sigma_i$ are $0$-cells and the $\beta_i$ are $1$-cells. Because $K$ is a finite regular CW complex and the gradient path strictly descends through regular cells without repeating, every $1$-path must be finite and eventually terminate at some $0$-cell $v$. If $v$ were regular, it would have to be paired with some $1$-cell $\beta$ such that $v \subset \beta$ and $f(\beta) \preceq f(v)$, allowing the path to continue. Since the path terminates at $v$, no such $1$-cell exists, meaning $|H_f(v)| = 0$. Since vertices have no proper faces, $|T_f(v)| = 0$ is trivially true. Thus, $v$ is a critical $0$-cell.

To see that $1$-paths cannot split, suppose a path reaches a vertex $\sigma_i$. For the path to split, $\sigma_i$ would need to be paired with two distinct $1$-cells $\beta$ and $\beta'$ in $\mathcal{V}$, meaning both would belong to $H_f(\sigma_i)$. However, by Condition (1) of Definition \ref{def:mdm_function_on_regularCW}, $|H_f(\sigma_i)| \le 1$. Therefore, there is at most one outgoing arrow from any vertex, and the path has a unique continuation. Paths may merge, however, because multiple $1$-cells can share a common vertex without violating the MDM conditions.
\end{proof}

Now we show the restriction of an MDM function to a subcomplex is an MDM function.

\begin{lemma}\label{lem:restricted_MDM}
Let $K$ be a finite regular CW complex, $L$ be a subcomplex of $K$ and $f : K \to \mathbb{R}^k$ an MDM function on $K$. Then, the restriction of $f$ to $L$ is an MDM function. Moreover, if $\sigma\in L$ is a critical $p$-cell of $f$, then it is critical for the restriction function to $L$.
\end{lemma}

\begin{proof}
Let $g = f|_L$ be the restriction of $f$ to the subcomplex $L$. We verify the four conditions of Definition \ref{def:mdm_function_on_regularCW} for $g$.

For any cell $\sigma \in L$, the set of its cofaces in $L$ is a subset of its cofaces in $K$. Therefore, $H_g(\sigma) \subseteq H_f(\sigma)$. Since $f$ is an MDM function, $|H_f(\sigma)| \le 1$, which implies $|H_g(\sigma)| \le 1$. Thus, Condition (1) holds.

Because $L$ is a regular subcomplex, it is closed under taking faces. Hence, all faces of $\sigma$ in $K$ are also in $L$, meaning $T_g(\sigma) = T_f(\sigma)$. Since $|T_f(\sigma)| \le 1$, we have $|T_g(\sigma)| \le 1$, satisfying Condition (2).

For Condition (3), let $\beta \in L$ be a coface of $\sigma$ such that $\beta \notin H_g(\sigma)$. Then $\beta \notin H_f(\sigma)$ either. Since $f$ is an MDM function, $f(\beta) \succ f(\sigma)$, which means $g(\beta) \succ g(\sigma)$.

For Condition (4), let $\alpha \in L$ be a face of $\sigma$ such that $\alpha \notin T_g(\sigma)$. Since $T_g(\sigma) = T_f(\sigma)$, $\alpha \notin T_f(\sigma)$. Thus, $f(\alpha) \prec f(\sigma)$, meaning $g(\alpha) \prec g(\sigma)$.

Therefore, $g$ is an MDM function on $L$. Finally, if $\sigma \in L$ is a critical $p$-cell of $f$, we have $|H_f(\sigma)| = 0$ and $|T_f(\sigma)| = 0$. Since $H_g(\sigma) \subseteq H_f(\sigma)$ and $T_g(\sigma) = T_f(\sigma)$, it immediately follows that $|H_g(\sigma)| = 0$ and $|T_g(\sigma)| = 0$. Thus, $\sigma$ is critical for the restricted function $g$.
\end{proof}

In the following results, we prove that an MDM function on a subcomplex $L$ of a finite regular CW complex $K$ can be extended to an MDM function on $K$.

\begin{lemma}\label{lem:extended_MDM}
Let $K$ be a finite regular CW complex, $L$ be a subcomplex of $K$ and $f : L \to \mathbb{R}^k$ an MDM function on $L$. Then, $f$ can be extended to an MDM function on $K$ such that each cell in $K\setminus L$ is critical.
\end{lemma} 

\begin{proof}
Let $f : L \to \mathbb{R}^k$ an MDM function on $L$.  Define $g : K \to \R^k$ by
\[
g(\sigma) =
\begin{cases}
f(\sigma) & \text{if } \sigma \in L, \\
(c_1 + \dim\sigma,\, c_2 + \dim\sigma,\,
\ldots,\, c_k + \dim\sigma) & \text{if } \sigma \notin L,
\end{cases}
\]
where $c_i = \max\{f_i(\tau) \mid \tau \in L\} + 1$ for each
$1 \le i \le k$.

We verify Definition \ref{def:mdm_function_on_regularCW} for $g$ and any $p$-cell $\sigma\in K$ such that $\alpha^{(p-1)}\subset \sigma$ and $\beta^{(p+1)}\supset\sigma$.

First suppose that $\sigma\in L$. Then $g(\sigma)=f(\sigma)$, and either $\beta \in L$ or $\beta\in K\setminus L$. So $g(\beta)=f(\beta)$ or $g(\beta)\succ g(\sigma)$ by definition of $g$. Then $H_g(\sigma)=H_f(\sigma)$. Since $f$ is an MDM function,  Condition (1) and Condition (3) are satisfied. 
Now consider $\alpha$. Since $L$ is a subcomplex, then $\alpha\in L$. So, $g(\alpha)=f(\alpha)$ and $T_g(\sigma)=T_f(\sigma)$. Since $f$ is an MDM function, Condition (2) and Condition (4) are satisfied. 

Now suppose that $\sigma\notin L$. Since $L$ is a subcomplex of $K$, $\beta\notin L$. Then, by definition of $g$, $g(\beta)\succ g(\sigma)$. So $H_g(\sigma)=\emptyset$, and Condition (1) and Condition (3) are satisfied. For any face $\alpha$, $g(\alpha)\prec g(\sigma)$ by definition of $g$. So, $T_g(\sigma)=\emptyset$, and Condition (2) and Condition (4) are satisfied. 

Therefore, $g$ is an MDM function on $K$ such that each cell $\sigma\in K\setminus L$ is a critical for $g$.

\end{proof}  

\begin{lemma}\label{lem:extended_MDM_oncollapse}
Let $K$ be a finite regular CW complex, $L$ be a subcomplex of $K$ such that $K$ collapses to $L$ and $f : L \to \mathbb{R}^k$ be an MDM function on $L$. Then $f$ can be extended to a MDM function on $K$ such that there is no critical cell in $K\setminus L$.
\end{lemma} 

\begin{proof}
Let $f : L \to \mathbb{R}^k$ be the MDM function on $L$. Suppose that $K$ collapses to $L$. Since $K$ is a finite regular CW complex, there exists a finite number of collapse pairs in $K\setminus L$. 
Let $(\sigma_1,\tau_1), (\sigma_2,\tau_2), \ldots, (\sigma_n,\tau_n) $ be a finite sequence of collapse pairs in $K\setminus L$ such that $\dim\tau_j=\dim\sigma_j+1$, $L_j=L_{j-1}\sqcup \{\sigma_{j}, \tau_{j}\}$ for each $1\leq j \leq n$ and $L_0=L$ and $L_n=K$. 

Define a function $g_j:L_j\to \mathbb{R}^k$ by
$$ g_j(\gamma)=
\begin{cases}
g_{j-1}(\gamma) & \text{if} \  \gamma\in L_{j-1}, \\
\textbf{c}_{j-1} & \text{if} \ \gamma=\sigma_{j},\tau_{j}
\end{cases} $$
where $g_0=f$ and $\textbf{c}_{j-1}=(c_{(j-1)1}, c_{(j-1)2}, \ldots, c_{(j-1)k})$ such that $c_{(j-1)i}=max\{g_{(j-1)i}(\gamma) \ | \ \gamma\in L_{j-1}\}+1$ for each component function $g_{(j-1)i}$ of $g_{j-1}$. 

Now we show that $g_j$ is an MDM function on $L_j$ by induction on $j$ such that neither $\sigma_j$ nor $\tau_j$ is critical for $g_j$.

For $j=1$, 
$$ g_1(\gamma)=
\begin{cases}
f(\gamma) & \text{if} \  \gamma\in L, \\
\textbf{c}_0 & \text{if} \ \gamma=\sigma_{1},\tau_{1}\\
\end{cases} $$
Now, we verify Definition \ref{def:mdm_function_on_regularCW} for $g_1$. Let $\gamma$ be any $p$-cell in $L_1$. 

If $\gamma\in L$, then $g_1(\gamma)=f(\gamma)$, $f(\gamma) \prec \textbf{c}_0=g_1(\sigma_1)$ and $g_1(\gamma)\prec \textbf{c}_0=  g_1(\tau_1)$ by definition of $g_1$. So, $\sigma_1, \tau_1\notin H_{g_1}(\gamma) $ and $H_{g_1}(\gamma) = H_f(\gamma)$. 
Moreover, since $L$ is a subcomplex of $K$, neither $\sigma_1$ nor $\tau_1$ is face of $\gamma$. So, $\sigma_1, \tau_1\notin T_{g_1}(\gamma)$ and $T_{g_1}(\gamma) = T_f(\gamma)$. Since $f$ is an MDM function on $L$, each condition given in Definition \ref{def:mdm_function_on_regularCW} is satisfied. In addition, regular cells and critical cells of $f$ in $L$ are also regular and critical for $g_1$, respectively. 

Now suppose $\gamma\notin L$. Then $\gamma=\sigma_1$ or $\gamma=\tau_1$. if $\gamma=\sigma_1$, then $\tau_1$ is the unique proper coface of $\gamma$ in $L_1$ because $(\sigma_1,\tau_1)$ is a collapse pair in $L_1$. But, by definition of $g_1$, $g_1(\sigma_1)=g_1(\tau_1)$. So, $H_{g_1}(\sigma_1)=\{\tau_1 \}$. Also, any proper face $\alpha$ of $\sigma_1$ is in $L$. So, by definition of $g_1$, $g_1(\alpha)\prec g_1(\sigma_1)$, which implies that $T_{g_1}(\sigma_1)=\emptyset$. If $\gamma=\tau_1$, then $H_{g_1}(\tau_1)=\emptyset$ because there is no proper coface of $\tau_1$ in $L_1$ by the collapse property. Also, for any proper face $\alpha\in L_1$ such that $\alpha\neq \sigma_1$, $g_1(\alpha)\prec g_1(\tau_1)$. Then $\alpha\notin T_{g_1}(\tau_1)$. However, $H_{g_1}(\sigma_1)=\{\tau_1 \}$ implies that $\sigma_1\in T_{g_1}(\tau_1)$. So, $T_{g_1}(\tau_1)=\{\sigma_1 \}$.  Thus, each condition given in Definition \ref{def:mdm_function_on_regularCW} is satisfied. Also $H_{g_1}(\sigma_1)=\{\tau_1 \}$ and $T_{g_1}(\tau_1)=\{\sigma_1 \}$ implies that $\sigma_1,\tau_1 \in K\setminus L$ are not critical for $g_1$. 

For any $1\leq j < n$, suppose $g_j$ is an MDM function on $L_j$ such that neither $\sigma_j$ nor $\tau_j$ is critical for $g_j$. By a similar argument given for $g_1$, we see that the function $g_{j+1}$ is an MDM function on $L_{j+1}$ such that neither $\sigma_{j+1}$ nor $\tau_{j+1}$ is critical for $g_{j+1}$. In addition, regular cells and critical cells of $g_j$ in $L_j$ are also regular and critical for $g_{j+1}$, respectively. 

As a consequence, we obtain that $g_n:L_n\to \mathbb{R}^k$ is an MDM function on $K=L_n$ satisfying the desired properties. 
\end{proof}

\begin{remark}\label{rmk:extension_regularCW}
Before concluding this section, we want to highlight that all the results given for simplicial complexes in \cite{Brouillette2024} and concerning the definition of an MDM function on simplicial complexes can be generalized to finite regular cell complexes by using Definition \ref{def:mdm_function_on_regularCW}. Also, the notions of isolated invariant sets, Conley index, Morse decompositions and the related results given \cite{Brouillette2024} work for the finite regular CW complexes thanks to \cite{Mrozek2017}. Therefore, we state the results in following sections for finite regular CW complexes without further modifications. 
\end{remark}

\subsection{Topological Dynamics and the Morse Equation}
In classical discrete Morse theory, the Morse inequalities are established by counting individual critical cells.  Because the flow of an MDM function is multivalued, we need to approach it through topological dynamics, specifically by using discrete versions of the Conley index~\cite{Batko2020}. The multiparameter flow $\Pi_f$ of an MDM function is the multivalued dynamical system generated by the gradient paths of $f$.    

\begin{definition}[\cite{Brouillette2024}]
\label{def:isolated_invariant}
Let $K$ be a finite regular CW complex, $f : K \to \R^k$ be an MDM
function on $K$, and $S$ be a subset of cells in $K$.
\begin{enumerate}
  \item $S$ is an invariant set with respect to $\Pi_f$ if
        any complete gradient trajectory that passes through $S$ is
        entirely contained within $S$.
  \item $S$ is an isolated invariant set if there exists a
        closed subcomplex $N \supseteq S$, called an
        isolating neighbourhood of $S$, such that every
        complete gradient trajectory lying entirely within $N$ is
        contained in $S$.
\end{enumerate}
\end{definition}

Within MDM theory, a critical component naturally acts as an
invariant set: it contains both its critical cells and the
internal gradient paths that connect them. It has been shown 
that each Morse set $M([\sigma])$ is in fact an isolated 
invariant set (see~\cite[Theorem~9.4]{Mrozek2017}).
  
Following Brouillette et al. \cite{Brouillette2024}, for the isolated invariant sets that arise as Morse sets in our setting, we can explicitly construct this index pair. The isolating neighborhood is simply the closure $N = \mathrm{Cl}(S)$ such that the exit set, $L = \mathrm{Ex}(S) = \mathrm{Cl}(S) \setminus S$ through which the forward flow leaves the neighborhood is a regular cell complex. 

To clarify this definition geometrically, let us consider the MDM function 
given in Figure~\ref{fig:MDMfunction}. Define
\[
  S_1 = \bigl\{f^{-1}((8,7)),\, f^{-1}((5,5)),\, f^{-1}((5,2)),\, 
               f^{-1}((5,3))\bigr\}.
\]
The exit set $\mathrm{Ex}(S_1) = \bigl\{f^{-1}((8,8)),\, f^{-1}((3,3)),\, 
f^{-1}((5,0))\bigr\}$ is a simplicial complex, so the first condition is 
satisfied. However, $S_1$ fails to be an invariant set: the gradient path 
passing through $f^{-1}((8,7)) \in S_1$ is not entirely contained within 
$S_1$ (it exits into $\mathrm{Ex}(S_1)$ and continues beyond). Since $S_1$ 
is not even invariant, it cannot be an isolated invariant set.

Now expand the set to include the top triangle:
\[
  S_2 = S_1 \cup \bigl\{f^{-1}((8,8))\bigr\}.
\]
The exit set shrinks to $\mathrm{Ex}(S_2) = \bigl\{f^{-1}((3,3)),\, 
f^{-1}((5,0))\bigr\}$. One can verify that every gradient path passing 
through a cell in $S_2$ is entirely contained within $S_2$, so $S_2$ 
is invariant. Moreover, no path can remain indefinitely in any isolating 
neighborhood of $S_2$ without belonging to $S_2$ itself. Hence $S_2$ is 
an isolated invariant set.

To capture the local topology of such an isolated invariant set $S$, one uses this index pair $(N, L) = (\mathrm{Cl}(S), \mathrm{Ex}(S))$. Following \cite[Definition~3.8]{Brouillette2024}, the discrete Conley index
of $S$ is defined as the relative homology $H_*(N, L;\mathbb{Z})$. The $p$-th Conley coefficient of $S$, denoted $\mathrm{con}_p(S)$, is the rank of the $p$-th homology group $H_p(N, L)$.

\begin{definition} \label{def:reachability}
Let $\mathcal{V}$ be the gradient vector field of an MDM function on a finite regular CW complex $K$. For arbitrary cells $\alpha,\beta\in K$, we write $\alpha \to_{\mathcal{V}} \beta$ if there exists a finite alternating sequence
\[
\alpha \supseteq \sigma_0 \subset \tau_0 \supset \sigma_1 \subset \tau_1 \supset \cdots \supset \sigma_r \subseteq \beta
\]
such that $(\sigma_i,\tau_i)\in \mathcal{V}$ for every $i$. In all subsequent definitions, the reachability relation between arbitrary cells is understood in this sense.
\end{definition}

\begin{definition}\label{def:V_compatible}
Let $\mathcal{V}$ be the gradient vector field of an MDM function on $K$ and $S\subseteq K$. We say that $S$ is a $\mathcal{V}$-compatible subset of $K$ if, for each $\sigma\in S$ such that either $(\sigma, \tau)\in \mathcal{V}$ or $(\nu, \sigma)\in \mathcal{V}$, then $\tau\in S$ or $\nu\in S$, respectively. 
\end{definition} 

Observe that the restriction of an MDM function $f$ to a $\mathcal{V}$-compatible subcomplex $S$ of $K$ is an MDM function on $S$ with no additional critical cells. 

To formally handle the multidimensional flow, we must establish when two critical cells belong to the same topological feature.

\begin{definition}[\cite{Brouillette2024}] \label{def:dynamically_connected}
Let $\mathcal{C}$ be the set of all critical cells of an MDM function $f:K\to\mathbb{R}^k$. For any $\sigma,\tau\in\mathcal{C}$, we say that $\sigma$ is parametrically flow-connected to $\tau$, denoted $\sigma \to_f \tau$, if $\sigma \to_{\mathcal{V}} \tau$ and their function values share at least one coordinate (i.e., $f_i(\sigma) = f_i(\tau)$ for some parameter index $i \in \{1,\dots,k\}$).
\end{definition}

The parametric flow-connectedness $\sim$ on $\mathcal{C}$ is defined as the equivalence relation generated by $\to_f$. Thus, $\sigma \sim \tau$ if there exists a sequence of critical cells $\sigma = \sigma_0, \sigma_1, \dots, \sigma_m = \tau$ such that for each $j$, either $\sigma_j \to_f \sigma_{j+1}$ or $\sigma_{j+1} \to_f \sigma_j$. The equivalence classes $[\sigma] \in \mathcal{C}/{\sim}$ are precisely the critical components of the MDM function.

Because the multiparameter flow $\Pi_f$ is a multivalued dynamical system, these critical components cannot be analyzed via single isolated cells. Instead, each critical component acts as an isolated invariant set with respect to the flow. Following \cite{Brouillette2024}, we associate to each critical component $[\sigma]$ a Morse set $M([\sigma])$. The local topology of this isolated invariant set is captured by its discrete Conley index, defined as the relative homology $H_*(N, L; \mathbb{Z})$ of an isolating neighborhood pair $(N, L)$.

\begin{example}
\label{ex:components}
To illustrate the nature of critical components under the dynamical connectedness relation $\sim$, we consider the $2$-dimensional simplicial complex $K$ and the MDM function $f:K\to \mathbb{R}^2$, given in Figure \ref{fig:MDMfunction}.
\begin{figure}[htb]
\includegraphics[width=0.5\textwidth]{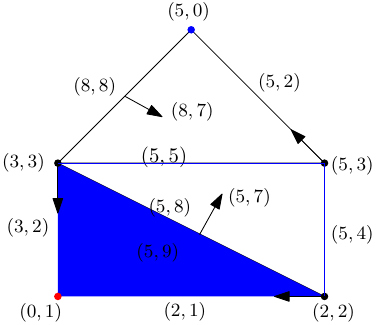}
\caption{Critical components of the MDM function given in Figure \ref{fig:MDMfunction}.}
\label{fig:Multicomponent}
\end{figure}
The set of critical cells of $f$ is $$C=\{ f^{-1}((0,1)), f^{-1}((5,0)), f^{-1}((5,4)), f^{-1}((5,5)), f^{-1}((5,9)) \}.$$ Let us denote these critical cells as $v_1=f^{-1}((0,1))$, $v_2=f^{-1}((5,0))$, $\sigma_1=f^{-1}((5,4))$, $\sigma_2=f^{-1}((5,5))$ and $\tau=f^{-1}((5,9))$.
Now we find the critical components of $f$ under the relation $\sim$. 

\begin{enumerate}[leftmargin=*]
    \item The Isolated Component (Red in Figure \ref{fig:Multicomponent}): There is a single critical vertex $v_1$  with the function value $f(v_1) = (0,1)$. Because no other critical cell in the entire complex has either a $0$ in its first parameter coordinate or  a $1$ in its second parameter coordinate, it is impossible for $v_1$ to satisfy the shared-parameter condition. Therefore, it cannot be dynamically connected to anything else and $[v_1] = \{v_1\}$.
    
    \item The Disconnected Component (Blue in  Figure \ref{fig:Multicomponent}): In the remainder of the complex, there are multiple other critical cells, including a critical triangle $\tau$ with value $f(\tau) = (5, 9)$, two critical edges  $\sigma_1$ and $\sigma_2$ with values $f(\sigma_1) = (5, 4)$ and $f(\sigma_2) = (5, 5)$, and a critical vertex $v_2$ with value $f(v_2) = (5, 0)$. 
    
    Geometrically, $v_2$ and other blue critical cells do not touch; they share no faces or vertices. In classical topological terms, they are completely disconnected. However, under the MDM framework:
    \begin{itemize}
        \item Flow Condition: The gradient field $\V$ provides a sequence of paired cells (a flow path) strictly descending from the triangle $\tau$ down through the regular interior cells, eventually reaching the edges $\sigma_1$ and $\sigma_2$. Thus, $\tau\to_f \sigma_1$ and $\tau\to_f \sigma_2$. Also, a sequence of paired cells strictly descending from the edges $\sigma_1$ and $\sigma_2$ down through the regular cells, eventually reaching the vertex $v_2$. Similarly, $\sigma_1\to_f v_2$ and $\sigma_2\to_f v_2$. 
        \item Parameter Condition: All of the blue critical cells share the exact same first parameter value $5$.
    \end{itemize}
    Because both conditions are met, we have $\tau \sim \sigma_1$, $\tau \sim \sigma_2$, $\sigma_1 \sim v_2$ and $\sigma_2 \sim v_2$.  Since the relation is an equivalence relation, we obtain $[\tau] = \{\tau, \sigma_1, \sigma_2, v_2\}$.
\end{enumerate}
By transitivity, several such geometrically separated critical cells merge into a single, unified critical component such as $[\tau]$. This highlights a crucial paradigm shift in the multiparameter setting: critical components are defined purely by their dynamical interaction with the gradient field across shared parameter values, not by their static adjacency in the simplicial complex. Consequently, when performing surgeries on connected sums (which will be detailed in our subsequent work), one must excise these entire dynamically linked Morse sets, regardless of their geometric footprint.
\end{example}

\begin{definition}[\cite{Brouillette2024}] \label{def:f_cycle}
Let $f: K \to \mathbb{R}^k$ be a MDM function on a finite regular CW complex $K$. A sequence of critical components $[\sigma_0], [\sigma_1], \dots, [\sigma_n] = [\sigma_0] \in \mathcal{C}/{\sim}$ is called an $f$-cycle if, for each $i \in \{1, 2, \dots, n\}$, there exist some $\tau'$ and $\tau$ such that $\tau' \sim \sigma_{i-1}$, $\tau \sim \sigma_i$ and $\tau' \to_f \tau$. If there is no $f$-cycle, we say $f$ is acyclic.
\end{definition}

\begin{remark} \label{rem:mdm_cycles}
It is worth noting why the concept of an $f$-cycle is necessary in the multidimensional setting. In classical discrete Morse theory ($k=1$), every gradient path induces a  decreasing sequence of real numbers, making  cycles mathematically impossible. However, the multiparameter gradient flow $\Pi_f$ is a multivalued dynamical system. Because the connectedness relation $\to_f$ only requires cells to share at least one parameter coordinate, a sequence of flow paths can loop back upon itself by alternating which coordinate remains constant across different steps. Consequently, an MDM function can possess $f$-cycles, making acyclicity a non-trivial condition that must be explicitly required.

Crucially, an acyclic MDM function still allows gradient paths to completely exit a critical component. If a path leaves a Morse set $M([\sigma])$ and  descends into a different Morse set without ever returning, this does not form an $f$-cycle. On the contrary, this downward leaking of the flow between components is the natural topological behavior that establishes the Morse decomposition.
\end{remark}

\begin{definition}[\cite{Brouillette2024}] \label{def:Morse_set}
 Let $f: K \to \mathbb{R}^k$ be an acyclic MDM function and $[\sigma]$ be a critical component of $f$. The Morse set of $[\sigma]$ is
$$
M([\sigma]) = \bigcup_{\tau,\tau’\in[\sigma]} C(\tau’,\tau),
$$
where $C(\tau’,\tau) = \bigl\{\beta\in K \;\bigm|\;
\tau’\to_\mathcal{V}\beta\to_\mathcal{V}\tau
\bigr\}$ is the set of all cells lying on $\mathcal{V}$-paths
from $\tau’$ to $\tau$.
\end{definition} 

For an acyclic MDM function, the set of critical components $\C/\sim$ induces a Morse decomposition of the complex $K$ (see   \cite[Theorem 7.8]{Brouillette2024}). The contribution of each critical component $[\sigma]$ to the global topology is measured by the Conley coefficients of its associated Morse set $M([\sigma])$. 

This yields the generalized Morse equation:
\begin{equation}
\label{eq:morse_poly}
\sum_{p=0}^{\dim K} m_p t^p = P_K(t) + (1+t)Q(t)
\end{equation}
Here, $P_K(t) = \sum \beta_p(K) t^p$ is the Poincar\'{e} polynomial of $K$. The coefficient $m_p = \sum_{[\sigma] \in \C/\sim} \con{p}{M([\sigma])}$ represents the total topological index in dimension $p$ across all critical components. 

The term $Q(t)$ is a polynomial with non-negative integer coefficients. Because the coefficients of $Q(t)$ are non-negative, equating the coefficients of $t^p$ on both sides of Equation (1) immediately yields the generalized Morse inequalities:
$$ m_p \ge \beta_p(K) \quad \text{for all } p \ge 0. $$

Geometrically, $Q(t)$ measures the extra critical structures, the ones that appear in the flow but do not actually contribute to the Betti numbers of the manifold. In classical theory, a non-zero $Q(t)$ indicates the presence of critical cells that could theoretically be canceled by gradient path modifications. In the MDM setting, $Q(t)=0$ indicates that our critical components represent the manifold's homology perfectly, with no unnecessary invariant sets involved.

\begin{lemma} \label{lem:conley_bound}
Let $f:K \to \mathbb{R}^k$ be an MDM function. Then the $p$-th Conley coefficient of the Morse set $M([\sigma])$ of a critical component $[\sigma]$ is bounded above by the number of critical $p$-cells in $M([\sigma])$.
\end{lemma}

\begin{proof}
The $p$-th Conley coefficient of the Morse set $M([\sigma])$ is $\con{p}{M([\sigma])}$. Let us denote the isolating neighborhood pair as $N = \mathrm{Cl}(M([\sigma]))$ and $L = \mathrm{Ex}(M([\sigma]))$. Then $\con{p}{M([\sigma])}$ is defined as the rank of the relative homology group $H_p(N, L)$.

By the foundational results of MDM theory \cite[Corollary 4.10]{Brouillette2024}, there exists a classical discrete Morse function $g: K \to \mathbb{R}$ that shares the exact same critical cells and gradient pairings (the vector field $\mathcal{V}$) as the MDM function $f$. Because the Morse set $M([\sigma]) = N \setminus L$ is an isolated invariant set, it is  a $\mathcal{V}$-compatible subset of $K$. By the definition of $\mathcal{V}$-compatibility, any gradient pair in $\mathcal{V}$ containing a cell in $N \setminus L$ must have its partner also strictly within $N \setminus L$. 

Consequently, the classical discrete Morse function $g$ induces no gradient pair with one cell in $L$ and the other in $N \setminus L$. Let $h$ be the restriction of $g$ to $N$. Because no gradient pairs cross the relative boundary between the interior and the exit set, any cells located on the exit set $L$ are either critical with respect to $h$, or they are paired with other cells strictly within $L$. Therefore, the critical $p$-cells of $h$ that lie strictly within the interior $M([\sigma])$ are exactly the critical $p$-cells of $g$ in $M([\sigma])$.

By classical discrete Morse theory \cite[Theorem 7.1]{Forman1998} applied to relative pairs, the relative homology $H_p(N, L)$ is generated by the relative Morse complex, whose basis elements are precisely these critical $p$-cells contained in $M([\sigma])$. Therefore, the $p$-th Conley coefficient $\con{p}{M([\sigma])}$ cannot exceed the total number of critical $p$-cells in $M([\sigma])$.
\end{proof}

\begin{remark}
In the smooth setting, the Pareto set of a vector-valued function
$f: M \to \R^k$ consists of points where the negative gradients of the
component functions share no common descent direction
\cite{Smale1975, Wan1975}. Brouillette \cite{Brouillette2025} formalizes a
discrete analogue: a cell $\sigma \in K$ is Pareto critical if it
belongs to a connected component $C$ of the level set $L_{f(\sigma)}$ such
that $H_*(\mathrm{Cl}\,C, \mathrm{Ex}\,C) \neq 0$. The critical components
of an MDM function are closely related to the connected components of this
discrete Pareto set; in particular, a refined equivalence relation
$\sim$ \cite[Definition~6.1]{Brouillette2025} groups critical cells in
strict accordance with Pareto components, which is useful when computing
critical components algorithmically to avoid artifacts from perturbation.
\end{remark}

\section{Defining Perfect MDM Functions}
\label{sec:perfect_def}

Brouillette \cite{Brouillette2025} introduces a notion of perfection for MDM
functions based on the Pareto set structure. In this section, we propose a
complementary notion, component-perfection, which is formulated in
terms of Conley coefficients and is better suited for the connected sum
operations that we will study in a subsequent paper. We compare the two 
notions (strictly perfect versus component-perfect), prove that the former 
implies the latter, and discuss the existence of both classes.

We propose two natural extensions of perfection to the multiparameter setting.

\begin{definition}[\cite{Brouillette2025}]\label{def:strictly_perfect}
An acyclic MDM function $f: K \to \mathbb{R}^k$ on a finite regular CW complex $K$ is strictly perfect if the total number of critical $p$-cells exactly equals $\beta_p(K)$ for all $p$. 
\end{definition}

\begin{remark}
While contemporary literature (e.g., \cite{Brouillette2025}) often treats acyclicity as an implicit  property of perfect gradient vector fields, we explicitly put the acyclicity condition here. Because MDM functions can admit $f$-cycles independently of individual coordinate descent, explicitly demanding acyclicity prevents  cases where a function possesses optimal critical cells but fails to induce a well-defined Morse decomposition. This  ensures the theoretical validity of the Morse sets and their subsequent Conley index computations.
\end{remark}

\begin{definition} \label{def:component_perfect}
An acyclic MDM function $f:K \to \mathbb{R}^k$  on a finite regular CW complex $K$ is component-perfect if $m_p = \beta_p(K)$ for all $p \ge 0$. Equivalently, the polynomial $Q(t) = 0$ in the MDM Morse equation.
\end{definition}

The component-perfect definition allows a critical component to contain multiple cells, provided the discrete Conley index of its associated Morse set $M([\sigma])$ either matches the relative homology of a single $p$-cell modulo its boundary, or is completely trivial across all dimensions. This reflects the true geometry of the multidimensional setting. Because MDM criticalities behave similarly to Pareto sets, a critical component can be isolated by the flow but still have no actual impact on the global topology.

When computing component-perfect MDM functions algorithmically, care must be taken in how critical components are partitioned. As noted by Brouillette \cite{Brouillette2025}, the original dynamical connectedness relation ($\sim$) can occasionally fracture single topological features into multiple components due to minor algorithmic perturbations in the function values. To rigorously isolate the Morse sets required for excision during the connected sum operation, it is often necessary to employ the refined equivalence relation $\sim$ \cite[Def. 6.1]{Brouillette2025}, which groups critical cells in strict accordance with the connected components of the multifiltering function's Pareto set.

\begin{theorem}
\label{thm:strict_implies_component}
Every strictly perfect MDM function is component-perfect, but the converse does not hold in general.
\end{theorem}

\begin{proof}
Let $f: K \to \mathbb{R}^k$ be strictly perfect, meaning it is an acyclic MDM function where the total number of individual critical $p$-cells is $c_p = \beta_p(K)$ for all $p$. 

By Lemma \ref{lem:conley_bound}, the $p$-th Conley coefficient of a Morse set is bounded above by the number of critical $p$-cells contained within it. Summing over all critical components $[\sigma] \in \mathcal{C}/{\sim}$, we have $m_p \le c_p$. 

Furthermore, because $f$ is an acyclic MDM function, it must satisfy the generalized Morse inequalities, which imply that $m_p \ge \beta_p(K)$. Combining these inequalities yields:
$$ \beta_p(K) \le m_p \le c_p = \beta_p(K) $$
Therefore, $m_p = \beta_p(K)$ for all $p$, which implies $f$ is component-perfect.

For the converse, it suffices to provide a counterexample. Example \ref{ex:component_perfect_not_strict} demonstrates an acyclic MDM function where $m_p = \beta_p(K)$ for all $p$, making it component-perfect; yet the total number of individual critical cells is greater than the corresponding Betti numbers of the complex. Thus, it is not strictly perfect.
\end{proof}

\begin{remark}
The strictly perfect condition is therefore the more rigid of the two. In particular,
a strictly perfect MDM function on $K$ induces an acyclic gradient field with
exactly $\sum_p \beta_p(K)$ fixed points.  In practice, this forces the component functions $f_1, \dots, f_k$ to share a common set of Morse pairings, a highly non-generic condition.
\end{remark}

\subsection{Strict vs. Component Perfection}
\label{subsec:strict_vs_component}
If every critical component consists of a single critical $p$-cell, its associated Morse set is precisely that cell, yielding a Conley index of $1$ in dimension $p$ and $0$ elsewhere. However, the converse is false. The component-perfect definition is strictly more general and, as we argue, much more natural for the multiparameter setting. 

In classical discrete Morse theory, each critical $p$-cell $\sigma$ is
dynamically isolated: it forms its own isolated invariant set with
respect to the gradient flow, and its Conley index
\[
  \mathrm{Con}(\{\sigma\}) = H_*(\mathrm{Cl}\,\sigma,\,\mathrm{Ex}\,\sigma)
  \;\cong\;
  \begin{cases} \mathbb{Z} & \text{in degree }p, \\ 0 & \text{otherwise,}\end{cases}
\]
has the homology of a single $p$-cell relative to its boundary. In the
multidimensional setting this geometric isolation of individual cell
is lost: critical cells cluster into components, and it is the component as
a whole, not any individual cell, that forms an isolated invariant set
with a well-defined Conley index. The Conley index of a critical component
is still nontrivial and still controls the topology; the difference is only
that it may now be carried collectively by several cells rather than by one.

\begin{example}\label{ex:component_perfect_not_strict}
 Let $f:K\to \mathbb{R}^2$ be the MDM function given in Example \ref{ex:components}.  The critical components of $f$ are $[v_1]$ and $[\tau] = \{\tau, \sigma_1, \sigma_2, v_2\}$. Moreover, the Morse set of $[v_1]$ is $M([v_1]) = \{v_1\}$ and the Morse set of $[\tau]$, $M([\tau])$, is the union of blue cells given in Figure~\ref{fig:Morsesets}. 
 \begin{figure}[htb]
\includegraphics[width=0.3\textwidth]{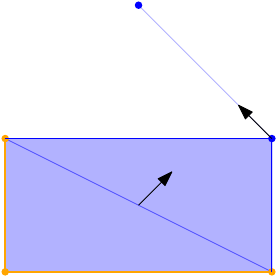}
\caption{Morse set (blue), $M([\tau])$, corresponding to the critical component $[\tau]$ in Figure \ref{fig:Multicomponent} and its exit set (orange), $\ex(M([\tau]))$.}
\label{fig:Morsesets}
\end{figure}
 Since $M([v_1]) = \{v_1\}$ is a simplicial complex, then $\ex(M([v_1])) = \emptyset$. Thus, 
$$
H_p(\cl(M([v_1])), \ex(M([v_1]))) = H_p(Cl(M([v_1])))
$$ 
and $H_p(\cl(M([v_1]))) = 0$ for $p \ge 1$ and $H_0(\cl(M([v_1]))) = \mathbb{Z}$. Then $\con{0}{M([v_1])} = 1$ and $\con{p}{M([v_1])} = 0$ for $p \ge 1$.
Now we will find $\con{p}{M([\tau])}$. Consider the index pair 
$$
(N, L) = (\cl(M([\tau])), \ex(M([\tau]))).
$$
Thanks to the following long exact sequence of the relative pair $(N, L)$,
$$ 
\dots \to H_p(L) \to H_p(N) \to H_p(N, L) \to H_{p-1}(L) \to \dots 
$$
Since $N$ and $L$ are contractible, $H_p(N) = H_p(L) = 0$ for $p \ge 1$ and $H_0(N) = H_0(L) = \mathbb{Z}$. So, $H_p(N, L) = 0$ for $p \ge 2$ and we have the following sequence:
$$ 0 \to H_1(N, L) \to \mathbb{Z} \to \mathbb{Z} \to H_0(N, L) \to 0 $$
Since $L \subseteq N$ and $N, L$ are connected, the map $\mathbb{Z} \to \mathbb{Z}$ is an isomorphism. Thus, we get $H_1(N, L) = H_0(N, L) = 0$. Then, $\con{p}{M([\tau])} = 0$ for $p \ge 0$.

Hence, $m_0 = \con{0}{M([v_1])} + \con{p}{M([\tau])} = 1 + 0 = 1$ and $m_p = \con{p}{M([v_1])} + \con{p}{M([\tau])} = 0 + 0 = 0$ for $p \ge 1$. Since the Betti numbers of the simplicial complex $K$ given in Figure~\ref{fig:Morsesets} are $\beta_0(K) = 1$ and $\beta_p(K) = 0$ for $p \ge 1$, $f$ is a component-perfect MDM. But $f$ is not strictly perfect because, for instance, the total number of critical $1$-cells of $f$, $c_1 = 2 > \beta_1(K) = 0$.
\end{example}

Example~\ref{ex:component_perfect_not_strict} highlights why we choose to work exclusively with component-perfect MDM functions. Demanding strict perfection is overly restrictive; it forces a multiparameter function to behave exactly like a $1$-dimensional function mapped symmetrically into $\R^k$. 

Conversely, component-perfection better reflects the multidimensional structure of the codomain. Recent algorithmic work demonstrates that when optimal MDM functions are generated from data, the critical cells naturally cluster into Pareto sets representing single topological features \cite{Brouillette2025}. 

\subsection{Existence of Perfect MDM Functions}
A natural question arises regarding the existence of perfect MDM functions. We can demonstrate that the class of component-perfect MDM functions on a manifold is non-empty whenever the manifold admits a classical perfect discrete Morse function.

Recall from classical discrete Morse theory that a standard perfect discrete Morse function $g: K \to \mathbb{R}$ exists for certain classes of manifolds, such as spheres, or can be obtained on specific triangulations of 3-manifolds under connected sum decompositions. Let $\mathcal{V}$ be the gradient vector field of such a function $g$, which is known to be acyclic.

For the strictly perfect case, the existence is theoretically straightforward but geometrically rigid. By \cite[Lemma 4.8]{Brouillette2024}, if we define a vector-valued function $f = (g, g, \dots, g): K \to \mathbb{R}^k$, then $f$ is an MDM function and its gradient vector field is identically $\mathcal{V}$. Because $g$ is a classical perfect discrete Morse function, the fixed points of $\mathcal{V}$ exactly match the Betti numbers of $K$. Consequently, $f$ has exactly $\beta_p(K)$ critical cells of index $p$, satisfying the condition for a strictly perfect MDM function. By Theorem \ref{thm:strict_implies_component}, $f$ is also component-perfect. This diagonal construction proves that the class of component-perfect MDM functions is non-empty.

While the diagonal construction is highly constrained due to the requirement that $H_f(\sigma) = \bigcap_{i=1}^k H_{f_i}(\sigma)$, it serves as a baseline. In practice, the component-perfect definition allows for much more flexible geometry. Recent algorithmic work provides motivation for this flexibility: when MDM functions are approximated from data, critical cells naturally cluster into Pareto sets representing single topological features. Therefore, since critical cells naturally group together into connected components in applications, component-perfection provides the right framework for our work on connected sums.

Now we show that the existence of a component-perfect MDM function on a closed, connected, orientable and triangulated manifold can be reduced to the existence of a component-perfect MDM function on its spine, where the spine of a closed $n$-manifold is defined as follows:

\begin{definition}\label{def:spine}
Let $M$ be a compact, connected and triangulated $n$-manifold with boundary and $N$ be a subcomplex of $M$ of dimension $d$ such that $d\leq n-1$. If $M$ collapses onto $N$ and there are no further collapses on $N$, then $N$ is called a spine of $M$. 

However, if $M$ is closed, then the spine of $M$ is defined as the spine of $M\setminus \sigma^{(n)}$ where $\sigma^{(n)}$ is an open $n$-cell. 
\end{definition}

\begin{theorem}\label{thm:component-perfect_on_spine}
Let $M$ be a closed, connected, orientable and triangulated $n$-manifold. If $M$ has a spine that admits a component-perfect MDM function $g$, then we can extend $g$ to a component-perfect MDM function $f$ on $M$. 
\end{theorem}

\begin{proof}
Let $N$ be the spine of $M$ that admits a component-perfect MDM function $g:N\to \mathbb{R}^k$. By Definition \ref{def:spine}, there exists an open $n$-cell $\sigma^{(n)}\in M$ such that $M\setminus \sigma^{(n)}$ collapses onto $N$. Then $M\setminus \sigma^{(n)}$ and $N$ are homotopy equivalent. So, $\beta_
p(M\setminus \sigma^{(n)})=\beta_
p(N)$ for each $0\leq p \leq n-1$. Also, by Mayer-Vietoris sequence for homology, $\beta_
p(M\setminus \sigma^{(n)})=\beta_
p(M)$ for each $0\leq p \leq n-1$.

Since $M\setminus \sigma^{(n)}$ collapses onto $N$, by Lemma \ref{lem:extended_MDM_oncollapse}, $g$ can be extended to an MDM function $g':M\setminus \sigma^{(n)}\to \mathbb{R}^k$ on $M\setminus \sigma^{(n)}$ such that the critical cells of $g'$ on $M\setminus \sigma^{(n)}$ are the critical cells of $g$ on $N$. So, the critical components of $g'$ are the critical components of $g$. Moreover, the Morse set of any critical component with respect to $g'$ coincides with the Morse set with respect to $g$. Since $g$ is component-perfect, then 

\begin{equation}\label{eq:conleycoefficient}
  m_p^{g'}=m_p^{g}=\beta_p(N)=\beta_p(M\setminus \sigma^{(n)})=\beta_
p(M)  
\end{equation}
for each $0\leq p \leq n-1$, where $m_p^{g'}$ (or $m_p^{g'}$)  denotes the $p$-th Conley coefficient with respect to $g'$ (or $g$). In addition, since $g$ is acyclic by Definition \ref{def:component_perfect}, then also $g'$ is acyclic.

Now, let us define a function $f:M\to \mathbb{R}^k$ by 
\[
f(\sigma) =
\begin{cases}
g'(\alpha) & \text{if } \alpha \in M\setminus \sigma^{(n)}, \\
(c_1,\, c_2,\,
\ldots,\, c_k) & \text{if } \alpha=\sigma^{(n)},
\end{cases}
\]
where $c_i = \max\{g'_i(\tau) \mid \tau \in M\setminus \sigma^{(n)} \} + 1$ for each component function $g'_i$.  Since $g^{\prime}$ is an MDM function on $M\setminus \sigma^{(n)}$, one can easily show that all conditions given in Definition~\ref{def:mdm_function_on_regularCW} are satisfied for $f$ and any cell in $M$. So, $f$ is an MDM function on $M$.  Since $f|_{ M\setminus \sigma^{(n)}}=g'$, by the definition of $f$, critical cells of $g'$ are also critical for $f$. Moreover, $\sigma^{(n)}$ is critical for $f$ such that, for each parameter index $i\in \{1,2,\ldots,k\}$, $f_i(\sigma)\neq f_i(\beta)$ for any critical cell $\beta$ of $f$ different from $\sigma^{(n)}$. So, $[\sigma^{(n)}]=\{\ \sigma^{(n)}\}$ and 
\[
  \mathrm{Con}(\{\sigma^{(n)}\}) = H_*(\mathrm{Cl}\,\sigma,\,\mathrm{Ex}\,\sigma)
  \;\cong\;
  \begin{cases} \mathbb{Z} & \text{in degree }n, \\ 0 & \text{otherwise,}\end{cases}
\]

Then, the Conley coefficient $m_p=0$ if $p\neq n$ and $m_n=1$ for $[\sigma^{(n)}]$ with respect to $f$.  So, by the equality given in \ref{eq:conleycoefficient}, we get 
$$ m_p^f=m_p^{g'}=\beta_p(M) \ \text{for} \  0\leq p \leq n-1, \ \text{and} \ m_n^f=1,$$ where $m_p^f$ denotes the $p$-th Conley coefficient with respect to $f$.
Since $M$ is closed and orientable by assumption, $\beta_n(M)=1$. Thus, $m_p^f=\beta_p(M)$ for each $0\leq p \leq n$.   Moreover, $f$ is acyclic thanks to the acyclicity of $g^{\prime}$ and the definition of $f(\sigma^{(n)})$.

Hence, $f$ is a component-perfect MDM function on $M$. 
\end{proof}

\bibliographystyle{alpha}

\end{document}